\documentclass{amsart}

\usepackage[foot]{amsaddr}
\usepackage[utf8]{inputenc}
\usepackage[T1]{fontenc}
\usepackage{lmodern}

\usepackage{amsmath,amsfonts,amssymb,amsthm}
\usepackage{stmaryrd}
\usepackage{dsfont}

\numberwithin{equation}{section}

\usepackage{microtype}
\usepackage{color}
\usepackage[hidelinks, colorlinks=true, citecolor=blue, linkcolor=black]{hyperref}
\usepackage{url}
\usepackage{booktabs}
\usepackage{nicefrac}
\usepackage{graphicx}
\usepackage{doi}

\newtheorem{theorem}{Theorem}[section]
\newtheorem{proposition}[theorem]{Proposition}
\newtheorem{lemma}[theorem]{Lemma}
\newtheorem{corollary}[theorem]{Corollary}
\newtheorem{assumption}[theorem]{Assumption}
\newtheorem{question}[theorem]{Question}

\theoremstyle{definition}
\newtheorem{definition}[theorem]{Definition}
\newtheorem{example}[theorem]{Example}

\newcommand{\N}{\mathbb{N}}
\newcommand{\Z}{\mathbb{Z}}
\newcommand{\Q}{\mathbb{Q}}
\newcommand{\R}{\mathbb{R}}

\newcommand{\T}{\mathbb{T}}
\renewcommand{\S}{\mathbb{S}}
\newcommand{\n}{\mathcal{N}}
\newcommand{\card}{\operatorname{card}}
\newcommand{\esp}[1]{\mathbb{E}\left[#1\right]}
\newcommand{\var}[1]{\operatorname{Var}\left(#1\right)}

\newcommand{\corr}[2]{\operatorname{corr}\left(#1,#2\right)}
\newcommand{\distZ}[1]{\operatorname{dist}\left(#1,\mathbb{Z}\right)}

\renewcommand{\cos}[1]{\operatorname{cos}\left(#1\right)}
\renewcommand{\sin}[1]{\operatorname{sin}\left(#1\right)}
\renewcommand{\exp}[1]{\operatorname{exp}\left(#1\right)}

\newcommand{\borel}[1]{\mathcal{B}\left(\mathcal{C}^{{#1}}(\overline{\Omega})\right)}
\newcommand{\Cyl}[1]{\operatorname{Cyl}^{#1}}

\title{Nodal replication of Planar Random Waves}
\author{Lo\"ic Thomassey}
\address{Laboratoire MAP5, Universit\'e Paris-Cit\'e}
\email{loic.thomassey@etu.u-paris.fr}

\author{Rapha\"el Lachi\`eze-Rey}
\email{raphael.lachieze-rey@parisdescartes.fr}
\begin{document}

	\maketitle
	\begin{abstract}
		We study the \textit{almost periods} of the eigenmodes of flat planar manifolds in the high energy limit. We prove in particular that the Gaussian Arithmetic Random Waves replicate almost identically at a scale at most  $\ell_{n}:= n^{-\frac{1}{2}}\exp{\n_n}$, where $\n_n$ is the number of ways $n$ can be written as a sum of two squares. It provides a qualitative interpretation of the  \textit{full correlation phenomenon} of the nodal length, which is known to happen at scales larger than $\ell_{n}':= n^{-1/2}\n_{n}^{A}.$ We provide also a heuristic with a toy model pleading that the minimal scale of replication should be closer to $\ell_{n}'$ than $\ell_{n}.$
	\end{abstract}
	
	\tableofcontents


\newcommand{\red}{\color{red}}
\section{Introduction}

This paper investigates random fields of the form
\begin{equation*}
	f(t)=\sum_{\lambda}a_\lambda e_\lambda(t)
\end{equation*}
where $t$ belongs either to the Euclidean space $\mathbb{R}^d$ or to the torus $\T^d$,
\begin{equation*}
	e_{\lambda}(t) := \exp{2i\pi \langle t,\lambda \rangle}, \quad t \in \T^{d},
\end{equation*}
and the $\lambda$ are finitely many deterministic wave vectors.\medbreak

A primary motivation is the zero set of the random Laplace eigenfunctions on the two-dimensional flat torus $\T^{2} = \R^{2} \slash \Z^{2}$. In this setting, the sequence of eigenvalues, or energy levels, are explicitly given by
\begin{equation*}
	E_{n} := 4\pi^{2}n,
\end{equation*}
where $n$ is expressible as a sum of two squares,
\begin{equation*}
	n \in \mathcal{S} := \{a^{2}+b^{2} : (a,b) \in \Z^{2}\}.
\end{equation*}
and the corresponding eigenspace $\mathcal{E}_{n}$ is  spanned by the $\mathbb{L}^{2}$-orthonormal Fourier basis, with $\lambda$ belonging to the set of frequencies
\begin{equation*}
	\Lambda_{n} := \{\lambda \in \Z^{2}:\lambda_{1}^{2}+\lambda_{2}^{2} = n \}.
\end{equation*}
It is a finite-dimensional vector space whose dimension
\begin{equation*}
	\mathcal{N}_{n} := \card \Lambda_{n}
\end{equation*}
is equal to the number of ways an integer can be written as a sum of two squares. The behaviour of $\mathcal{N}_{n}$ is well-understood and it is known that $n \in \mathcal{S}$ if and only if every prime divisor of $n$ congruent to $3$ modulo $4$ has an even valuation. In the latter case, $\mathcal{N}_{n}$ has an explicit formula related to the prime decomposition of $n$, that is,
\begin{equation}
\label{eq:decomp}
	\mathcal{N}_{n} = 4\prod_{i=1}^{k}(1+\alpha_{i}),
\end{equation}
where
\begin{equation*}
	n=2^{\alpha}p_{1}^{\alpha_{1}}...p_{k}^{\alpha_{k}}q_{1}^{2\beta_{1}}...q_{l}^{2\beta_{l}}
\end{equation*}
with $p_{i}\equiv 1 \pmod{4}$ and $q_{i} \equiv 3 \pmod{4}$ prime numbers.\medbreak

The preceding formula shows that the sequence $\mathcal{N}_{n}$ is subject
to large fluctuations. Yet, if we exclude some exceptional values, it is always possible to consider a subsequence of integers $n$ for which the following assumptions hold.\medbreak

	There is a density $1$ subset $  \mathcal  S  '\subset  \mathcal{S}$ of integers $n$   such that
	\begin{equation}\label{equation 1.1}
		\mathcal{N}_{n} = \log(n)^{\log(2)/2+o(1)},
	\end{equation}
	in the sense that for every $0 < \kappa < \frac{\log(2)}{2}$
	\begin{equation}\label{equation 1.2}
		\mathcal{N}_{n} \geq \log(n)^{\kappa}
	\end{equation}
	for sufficiently large  $n\in  \mathcal{S}' $, and such that the angular distribution
	\begin{equation*}
		\widetilde{\mu_{n}} := \frac{1}{\mathcal{N}_{n}}\sum_{\lambda \in \Lambda_{n}} \delta_{\lambda / \sqrt{n}}
	\end{equation*}
	converges to the Haar measure $\sigma$ on $\S^{1}$ for the weak-$\ast$ topology as $n \in \mathcal{S}'\to \infty $. Moreover, the Kolmogorov distance between $\sigma$ and $\widetilde{\mu_{n}}$ satisfies the bound
	\begin{equation}\label{equation 1.3}
		\operatorname{Kol}(\widetilde{\mu_{n}}, \sigma) := \sup_{\Gamma \text{arc on } \S^{1}} |\widetilde{\mu_{n}}(\Gamma)-\sigma(\Gamma)|\leq \frac{1}{\log(n)^{\rho}},
	\end{equation}
	for any $\rho < \frac{1}{2}\log(\pi/2)$  and $n$ large enough.\medbreak

For the interested reader, discussions related to these results are available respectively in \cite{BMW2017} and \cite[Theorem 1]{EH1999}. From now on, we we only consider $n\in \mathcal S'$ so that equations \eqref{equation 1.1}-\eqref{equation 1.3} are satisfied throughout the article.

\subsection{Planck scale Arithmetic Random Wave}

To study the properties of a generic Laplace eigenfunctions on the torus, it is possible to embed the previous model in a probability space. Notice first that every real-valued eigenfunction of the Laplace-Beltrami operator can be uniquely written as
\begin{equation*}
	T_{n}(t):=\frac{1}{\sqrt{\n_n}}\sum_{\lambda \in \Lambda_{n}} a_{\lambda} e_{\lambda}(t), \quad t \in \T^{2},
\end{equation*}
where $(a_{\lambda})$ is a sequence of complex numbers satisfying $a_{-\lambda}=\overline{a_{\lambda}}$. If the $a_{\lambda}$ are chosen as a realisation of independent and identically distributed complex standard random normal variables, then the resulting random field is an infinitely differentiable, centered, stationary Gaussian process with covariance function
\begin{equation*}
	r_{n}(t):=\frac{1}{\mathcal{N}_{n}}\sum_{\lambda \in \Lambda_{n}} \cos{2\pi\langle \lambda,t \rangle},\quad t\in \T^{2}.
\end{equation*}

As expected, it satisfies almost surely the partial differential equation
\begin{equation*}
	\Delta_{\T^{2}}T_{n}+E_{n}T_{n} = 0.
\end{equation*}

Alternatively, instead of describing the preceding field with its covariance function, it is sometimes preferable to work directly with its spectral measure
\begin{equation*}
	\mu_{n} := \frac{1}{\mathcal{N}_{n}} \sum_{\lambda \in \Lambda_{n}} \delta_{\lambda},
\end{equation*}
which, in that instance, is a purely atomic probability measure concentrated on the circle of radius $\sqrt{n}$.\medbreak

It is natural to rescale the previous Gaussian process by a factor $n^{-1/2}$, also known as the {\it Planck scale}. The notations $\widetilde{.}$ will from now on refer to the rescaled variables.
\begin{equation*}
	\begin{aligned}
		&\widetilde{T}_{n}(t) := T_{n}\left(\frac{t}{\sqrt{n}}\right), \quad t \in \sqrt{n}\T^{2},\\
		&\widetilde{r}_{n}(t) := r_{n}\left(\frac{t}{\sqrt{n}}\right), \quad t \in \sqrt{n}\T^{2},\\
		&\widetilde{\mu_{n}} := \frac{1}{\mathcal{N}_{n}}\sum_{\lambda \in \Lambda_{n}}\delta_{\frac{\lambda}{\sqrt{n}}}.
	\end{aligned}
\end{equation*}

As detailed later, the factor $n^{-1/2}$ can be viewed as the characteristic distance of the model. In domains whose diameter is below this threshold, $T_{n}$ will behave almost like an eigenvalue of the Laplacian $\Delta$ on $\R^{2}$ in the sense that $\widetilde{T}_{n}$ will converge (in some sense) to a Berry Random Wave $T$, that is the unique Gaussian process on $\R^{2}$ which is stationary, rotation-invariant, ergodic and which satisfies almost-surely the equation
\begin{equation*}
	\Delta T + 4\pi^{2} T = 0.
\end{equation*}

The covariance function associated with $T$ is the so-called Bessel function of the first kind,
\begin{equation*}
	\widetilde{r}(t):=J_{0}(2\pi|t|).
\end{equation*}
As a consequence of \eqref{equation 1.3},$ \widetilde{T}_{n}$ converges to $T$ on a scale slightly larger than the Planck scale: 
\begin{lemma}\cite[Proposition 5.1]{DNP2020} \label{lemma 1.3}
	Let $\alpha \in \N^{2}$, then there is some $\rho>0$ such that
	\begin{equation*}
		 \sup_{|t| \leq  \log(n)^{\rho}}|\partial^{\alpha}\widetilde{r}_{n}(t)-\partial^{\alpha}\widetilde{r}(t)| \longrightarrow 0.
	\end{equation*}
\end{lemma}

\subsection{Nodal set}

An object intensively studied in the literature is the zero set $T_n^{-1}\left(0\right)$, known as the nodal set or nodal lines. Almost surely, it is a one-dimensional smooth manifold whose length on a subdomain $\Omega \subset \T^{2}$ is denoted by
\begin{equation*}
	\mathcal{L}_{n}(\Omega) := \operatorname{len}\left(T_{n}^{-1}\left(0\right) \cap \Omega\right),
\end{equation*}
or simply
\begin{equation*}
	\mathcal{L}_{n,s} := \operatorname{len}\left(T_{n}^{-1}(0) \cap B(s)\right)
\end{equation*}
when $\Omega$ is $B(s)$, the centered ball of radius $s$.\medbreak

Oravecz, Rudnick, Wigman \cite{ORW} were able to compute the expected nodal length
\begin{equation*}
	\esp{\mathcal{L}_{n}(\Omega)} = \frac{|\Omega|}{2\sqrt{2}}\sqrt{E_{n}}.
\end{equation*}

Higher moments were harder to resolve as Kac-Rice formulas require in those cases a deep understanding of the behaviour of the correlation function $r_{n}$ and its derivatives, turning the preceding integral computation into an arithmetical problem intrinsically related to the distribution of the spectral measure $\mu_{n}$ alongside the circle of radius $\sqrt{n}$. It was only in 2011 that Krishnapur, Kurlberg and Wigman \cite{KMKW2011} were finally able to evaluate the asymptotic leading term of the variance of the nodal length on the full torus in the high energy limit :
\begin{equation*}
	\var{\mathcal{L}_{n}\left(\T^{2}\right)} \asymp \frac{1}{512}\frac{E_{n}}{\mathcal{N}_{n}^{2}}
\end{equation*}

The preceding variance is smaller than what was initially expected, but the result corroborates Berry's cancellation phenomenon, a principle related to the length of the nodal set of Gaussian Random Waves, which was first noticed by Berry in his seminal paper \cite{Be2002}. The computation of the variance on subdomains is still possible but it requires a more careful analysis. It was done in \cite{BMW2017}, an article in which the authors derived the following asymptotic
\begin{equation*}
	\var{\mathcal{L}_{n,s_{n}}} \asymp \frac{|B(s_{n})|}{512}\frac{E_{n}}{\mathcal{N}_{n}^{2}},
\end{equation*}
valid as long as the radius $s_{n}$ is  above Planck-scale, that is $s_{n}>n^{-1/2+\varepsilon}$ where $\varepsilon>0$. The preceding estimate immediately implies the full-correlation of the nodal length on the torus :

\begin{theorem}\label{theorem 1.1}
	For every $\varepsilon>0$,
	\begin{equation*}
		\sup_{s > n^{-1/2+\varepsilon}}\corr{\mathcal{L}_{n}}{\mathcal{L}_{n,s}} \longrightarrow 1.
	\end{equation*}
\end{theorem}

This unexpected theorem has a nice consequence. In the high energy limit, one may retrieve the full nodal length on the torus $\T^{2}$ by simply observing the nodal length in a window slightly larger than Planck-scale. It is natural to ask whether the previous theorem still holds below Planck-scale and the answer is negative (see Section \ref{sec:almost-per}). As of today, the phase transition between no correlation and full-correlation is suspected to happen at some $s = n^{-1/2}\log(n)^{A}$ and some lower and upper bounds for $A$ have been recently derived in \cite{Sa2020} and \cite{DNP2020}.

\subsection{Full correlation and replication}

The proof of theorem \ref{theorem 1.1} relies heavily on a careful study of the so-called spectral quasi-correlations and it requires an extensive use of deep arithmetical results regarding the angular distribution $\widetilde{\mu_{n}}$. An explanation of this phenomenon, and of why full correlation does not occur on the sphere, can be found in the work of Todino  \cite[Section 2.1]{Tod20}.\medbreak

The reader is encouraged to view this article as an alternative  qualitative explanation of Theorem \ref{theorem 1.1}. We show that not only the length of the nodal line, but the whole field, replicates almost identically at a scale $\tau_{n}$ that we try to estimate, and we advocate that it should be close to the scale of full correlation. It makes sense that if the field replicates almost identically at scale $\tau _{n}$, then by a continuity argument its nodal lines should replicate as well, and this is the purpose of Theorem \ref{proposition 2.3} and Section \ref{sec:lines}.

A rough explanation of this result is that replication occurs on the torus  because the dimension of the eigenspace on the torus is logarithmic, whereas it is polynomial on other manifolds, such as higher dimensional tori or the sphere (see Section \ref{sec:manifolds}).  We stress that our results do not provide another proof of full correlation, for which quasi-correlations are the right tool  \cite{BMW2017,Tod20}. Let us first explain what we mean by  {\it almost replication}.\\

\textbf{A sequence of almost periods} (of the Arithmetic Random Wave) is a sequence $\tau_{n} \in \T^{2}$ such that a.s.
\begin{equation}\label{equation 2.1}
	\sup_{t \in \T^{2}}|T_{n}(t+\tau_{n})-T_{n}(t)| \underset{n \to +\infty}{\longrightarrow} 0
\end{equation}
with the complementary condition
\begin{equation}\label{equation 2.2}
	\liminf \sqrt{n}|\tau_{n}| \geq 1.
\end{equation}

This latter requirement ensures that the aforementioned almost periods are far from $0$ with respect to the characteristic distance of the model. This is a necessary condition as nothing prevents otherwise a sequence of almost periods to be arbitrary close to $0$. 
Hence, upon the existence of such $\tau _{n}$, we shall say that the random waves  {\it almost replicate} at scale $\tau _{n}$, and we try in the rest of the paper to give an upper bound as small as possible on an almost period of $\tau _{n}$.

\medbreak

\subsubsection{Almost periodicity of eigenmodes}
   
As an introduction, and instead of studying directly full-correlation on the torus $\T^{2}$, one may start with a similar result on the $1$-dimensional torus $\T^{1}$. In this case, the set of eigenvalues consists of the energy levels $E_{n}=(2\pi)^{2}n^{2}$, $n \in \N$, and an orthonormal basis of the associated eigenspace is given by the standard Fourier basis $\cos{2\pi n t}$ and $\sin{2\pi nt}$. In these settings, the Random Wave Model is a periodic centered Gaussian process whose covariance function is $\cos{2\pi n t}$ and an explicit formula is given by
\begin{equation*}
	X_{n}(t) := a \cos{2\pi n t} + b \sin{2\pi n t}, \quad t \in \T^{1},
\end{equation*}
where $a,b$ are independent standard normal distribution.\medbreak

The characteristic distance of this model is $1/n$, which is also the smallest period of $T_{n}$. In this instance, full-correlation has a straightforward explanation which boils down to one word, \textit{periodicity}. In fact, if the zero set of $T_{n}$ is known on an open ball slightly larger than Planck scale, then it is also known on a full period of $X_{n}$, and \textit{a fortiori} on the whole torus. Last but not least, notice that the periodicity of $X_{n}$ is encoded by the underlying periodicity of its correlation function. \medbreak

It is very tempting to extend the previous reasoning to the torus $\T^{2}$ but unfortunately, the covariance function is no longer periodic at Planck scale. Yet, it is possible to circumvent in some ways this issue if the periodicity assumption is relaxed to almost periodicity. This idea is inspired by the study of almost periodic Gaussian processes in \cite{LR2020}.\medbreak

Say a function $f:\mathbb{R}^d\to\mathbb{C}$ is (Bohr-)almost periodic if and only if, for any $\varepsilon > 0$, there exists $T_{0}(\varepsilon)>0$ such that every Euclidean ball of radius $T_{0}(\varepsilon)$ contains an $\varepsilon$-almost period $\tau$ in the sense that
\begin{equation*}
	\sup_{t \in \R^{d}}|f(t+\tau)-f(t)| \leq \varepsilon
\end{equation*}
The covariance function $r_{n}$ of the Arithmetic Random Wave being a trigonometric polynomial, it is in particular an almost periodic function. Thus, if $\tau$ is a $\varepsilon$-almost period of $r_{n}$, then
\begin{equation*}
	r_{n}(\tau)=\corr{T_{n}(t)}{T_{n}(t+\tau)} \geq 1-\varepsilon.
\end{equation*}
$T_{n}(t)$ and $T_{n}(t+\tau)$ are hence heavily correlated and one might expect that if $t$ is a zero of $T_{n}$, then there is probably another zero in some neighbourhood of $t+\tau$, both the size of the neighbourhood and the probability depending on the precision $\varepsilon$. Hence, if the zeros of $T_{n}$ are known on a window $\Omega$, then they are also known on the translated windows $\tau +\Omega$, $2\tau +\Omega$ and so forth, at least with high probability. If the analogy with the $1$-dimensional torus $\T^{1}$ stills holds, the almost periods of $T_{n}$ must be of order slightly larger than Planck-scale to get full-correlation.

\section{Almost periodicity and replication}

The sections introduces the main results proved in this article, and the discussion around the estimation of the pseudo-period. First, we recall that the covariance function of the Arithmetic Random Wave is a periodic function on the lattice $\Z^{2}$. Our aim is to exhibit smaller  almost periods of the order slightly above Planck scale, that is slightly above $n^{-1/2}$.

\subsection{Existence of almost periods in the Arithmetic Random Wave}
\label{sec:almost-per}
The first result of this article is the following theorem related to the existence of almost periods.

\begin{theorem}\label{theorem 2.1}
	For $\alpha >0$, there exists a sequence $\tau_{n}$ of almost periods  such that asymptotically for $n\in \mathcal  S'$
	\begin{equation}\label{equation 2.1.1}
		|\tau_{n}| \leq \exp{\n_{n}^{1+\alpha}} n^{ -1/2}
	\end{equation}
	\medbreak
	and  the error term
	\begin{equation*}
		\varepsilon_{n} := \sup_{t \in \T^{2}} |r_{n}(t+\tau_{n})-r_{n}(t)|
	\end{equation*}
	is asymptotically bounded above by
	\begin{equation}\label{equation 2.1.2}
		\varepsilon_{n} \leq e^{-\log(n)^{\kappa}}
	\end{equation}
	for  $\kappa>0$ depending on $\alpha$.
\end{theorem}

The proof of Theorem \ref{theorem 2.1} relies on a well-known deterministic principle in Diophantine approximation, known as Dirichlet's approximation Theorem, which in our framework can be stated in the following manner.

\begin{theorem}[Dirichlet's approximation Theorem] \label{theorem 2.2}
	Let $\n\geq 1$   and consider an almost periodic function defined by
	\begin{equation*}
		f(t) = \frac{1}{\n}\sum_{k=1}^{\n} \cos{2 \pi \langle \gamma_{k}, t\rangle}, \quad t\in \R^{d}
	\end{equation*}
	with $\gamma_{k}$ some vectors in $\R^{d}$. \medbreak

	If $\varepsilon>0$, then there is an $\varepsilon$-almost period $\tau$, that is a vector $\tau \in \R^{d}$ for which
	\begin{equation*}
		\sup_{t \in \R^{d}} |f(t+\tau)-f(t)|\leq \varepsilon,
	\end{equation*}

	satisfying
	\begin{equation*}
		1 \leq |\tau| \leq (2\pi)^{\n/d}\varepsilon^{-\n/d}.
	\end{equation*}
\end{theorem}

We will now prove Theorem \ref{theorem 2.1} assuming we have already proved Dirichlet's approximation Theorem.

\begin{proof}
	Recall that
	\begin{equation*}
		r_{n}(t) = \frac{1}{\mathcal{N}_{n}}\sum_{\lambda \in \Lambda_{n}} \cos{2\pi \left \langle \lambda_{k}, t\right \rangle}.
	\end{equation*}
	Since $n\in \mathcal{S}'$,
	\begin{equation*}
		\mathcal{N}_{n} = \log(n)^{\log(2)/2 + o(1)}\geqslant \log(n)^{\kappa /\alpha }
	\end{equation*}
	for $n$ large enough and $\kappa <\alpha \log(2)/2$.
	Apply Theorem \ref{theorem 2.2} to $T_{n}$ and $\varepsilon_{n} = \exp{-\log(n)^{\kappa}}$. This ensures the existence of some almost period $\tau_{n} \in \R^{2}$ of $r_{n}$ that satisfies asymptotically the inequality
	\begin{align*}
		|\tau_{n}|&\leq (2\pi)^{\mathcal{N}_{n}/2}\varepsilon_{n}^{-\mathcal{N}_{n}/2}\\
		&= \exp{\frac{1}{2}\log(2\pi)\n_{n} + \frac 12 \log(n)^{\kappa}\n_{n}}\\
	&\leq \exp{\frac{1}{2}\log(2\pi)\n_{n} + \n_{n}^{1+\alpha/2}}\\
		&\leq e^{\n_{n}^{1+\alpha}}.
	\end{align*}
 	This finishes the proof.
\end{proof}

Notice that this latter result does not make use of any underlying arithmetical property satisfied by the wave vectors $\gamma_{k}$. This is in sharp contrast with the case of the Arithmetic Random Wave for which the wave vectors $\lambda_{k}$ have a deep number  theoretic flavour. This probably means that the bound obtained on the almost periods in Theorem \ref{theorem 2.1} is not optimal and could be improved, this is the topic of the remainder of the section.\medbreak

Let now $\tau_{n}$ be some sequence of almost periods as predicted by Theorem \ref{theorem 2.1} and denote hereafter
\begin{equation*}
	T_n'(t):= T_n(\tau_n+t), \quad t \in \T^{2},
\end{equation*}
the translated Gaussian process and
\begin{equation*}
	r_n'(t):= r_n(\tau_n+t), \quad t \in \T^{2},
\end{equation*}
its corresponding covariance function.\medbreak

The preceding theorem has a simple reformulation. In the high energy limit, $T_{n}(t)$ and $T_{n}'(t)$ are fully correlated and almost indistinguishable. The existence of almost periods for the correlation entails an approximate replication of the whole field at multiples of the so-called almost period. In fact, the Borell-TIS inequality  \cite[Theorem 2.1.1]{AT07}  yields that the supremum of the difference between the rescaled fields $\tilde T_{n}$ and $\tilde T_{n}'$  over the whole rescaled torus goes to $0$, as well as for their derivatives, meaning that \eqref{equation 2.1} is in order for some $\tau _{n}\leqslant \exp{\n_{n}^{1+\alpha }}$. This has some interesting geometric consequences for nodal lines, which will be discussed now.

\begin{theorem}\label{proposition 2.3}
	Let $\Omega$ be an open convex subset of $\R^{2}$ with compact closure and smooth boundary, say $\mathcal{C}^{1}$. Denote by
	\begin{equation*}
		\begin{aligned}
			&\widetilde{\mathcal{Z}}_{n}(\Omega) := \{t \in \Omega \,:\,\widetilde{T}_{n}(t) = 0\},\\
			&\widetilde{\mathcal{Z}}_{n}'(\Omega) := \{t \in \Omega \,:\,\widetilde{T}_{n}'(t) = 0\},
		\end{aligned}
	\end{equation*}
	the respective zero sets of $\widetilde{T}_{n}$ and $\widetilde{T}_{n}'$ on $\Omega$ and let $H_{1}(\mathrm{d}t)$ be the one-dimensional Hausdorff measure on $\R^{2}$ normalised so that any segment of length $1$ has Hausdorff measure $1$.\medbreak

	Then, for any continuous function $\varphi : \R^{2} \mapsto \R$ with support in $\Omega$,
	\begin{equation}\label{equation 2.2.1}
		\int_{\widetilde{\mathcal{Z}}_{n}(\Omega)} \varphi(t) H_{1}(\mathrm{d}t) - \int_{\widetilde{\mathcal{Z}}_{n}'(\Omega)} \varphi(t) H_{1}(\mathrm{d}t) \longrightarrow 0,
	\end{equation}
	the convergence holding in distribution.
\end{theorem}

In the high energy limit, the nodal sets of the Arithmetic Random Wave $T_{n}$ and its translate $T_{n}'$ are locally (that is at Planck scale) almost confounded and are geometrically very close. Therefore, the existence of almost periods for the covariance functions give rise to an asymptotic phenomenon of replication of the nodal lines.\medbreak

Last but not least, instead of fixing an open set $\Omega$ with compact support, it is also possible to take a sequence of slowly increasing open sets $\Omega_{n}$ if $\varphi $ has an unbounded support, under some decay assumptions. Unfortunately, it makes the proof less readable and enlightening and we decided to privilege clarity and concision.\medbreak

Theorem \ref{theorem 2.1} provides an upper bound on the existence of almost periods for the Arithmetic Random Waves. Similarly, there is a corresponding lower bound which is easily derived from  \eqref{equation 1.3}.

\begin{proposition}\label{theorem 2.4}
	Let $\tau_{n}$ be any sequence of almost periods of $r_{n}$. Then, there is some $A>0$ such that
	\begin{equation*}
		|\tau_{n}| \geq \frac{\log(n)^{A}}{\sqrt{n}}.
	\end{equation*}
	for $n$ large enough.
\end{proposition}

Contrary to  the upper bound proof which depends only on the number of eigenvalues, the lower bound relies intrinsically on a careful understanding of the angular distribution $\mu_{n}$. As such, it is only valid for the Arithmetic Random Wave model.

\begin{proof}
	According to the conclusion of Lemma \ref{lemma 1.3}, there is some $\rho>0$ such that
	\begin{equation*}
		\sup_{|t| \leq  \sqrt{n}\log(n)^{\rho}}|\widetilde{r}_{n}(t)-\widetilde{r}(t)| \longrightarrow 0.
	\end{equation*}

	Let $0 < A < \rho$ and assume that the conclusion of Proposition \ref{theorem 2.4} does not hold. Then, we can find some sequence of almost periods $\tau_{n}$ and an increasing sequence $n_{k}$ so that
	\begin{equation*}
		|\tau_{n_{k}}| \leq \frac{\log(n_{k})^{A}}{\sqrt{n_{k}}}.
	\end{equation*}
	In that case
	\begin{align*}
		|1-r_{n_{k}}(\tau_{n_{k}})|\geq& |1-J_{0}(2\pi \widetilde{\tau}_{n_{k}})|-|\widetilde{r}_{n_{k}}(\widetilde{\tau}_{n_{k}})-J_{0}(2\pi|\widetilde{\tau}_{n_{k}}|)|\\
		&\geq |1-J_{0}(2\pi\widetilde{\tau}_{n_{k}})| + o(1).
	\end{align*}

	But, as $|\widetilde{\tau}_{n_{k}}| \geq \log(n_{k})^{A} \longrightarrow +\infty$ and $J_{0}$ converges to $0$ at infinity,

	\begin{equation*}
		\lim_{k \to +\infty}|1-r_{n_{k}}(\tau_{n_{k}})| = 1.
	\end{equation*}

	This clearly contradicts the definition of $\tau_{n}$ as
	\begin{equation*}
		1-r_{n}(\tau_{n}) \longrightarrow 0.
	\end{equation*}
\end{proof}

\subsubsection{Discussion}   

If we combine Theorem \ref{theorem 2.1} and Proposition \ref{theorem 2.4}, we see that the smallest almost period $\tau_{n}$ is located in an intermediate range,  between $n^{-1/2}\log(n)^\rho$ and the subpolynomial range $n^{-1/2}\exp{\log(n)^{\kappa '}}$ for   $\kappa '\in (\log(2)/2,1)$. This is very interesting as it is suspected that the Arithmetic Random Wave exhibits some kind of phase transition at a range $n^{-1/2}\log(n)^{A}$ for some $A>0$. Below this threshold, $T_{n}$ will behave similarly to the Berry Random Wave Model as it proved by  Dierickx, Nourdin, Peccati and Rossi in \cite{DNP2020}. Above this threshold, full correlation of the nodal length appears and our random eigenfunctions can no longer behave like a Berry Random Wave. This observation leads to a natural question, that is whether the phenomenon of nodal replication highlighted in this article is linked with the phenomenon of full correlation of the nodal lines. If that was the case, it would imply that one can find sequences of almost periods in the logarithmic range $n^{-1/2}\log(n)^{A}$. Such a result would mean that the upper bound in theorem \ref{theorem 2.1} obtained via Dirichlet's approximation Theorem is far from optimal.
 That is not unlikely, yet it would be surprising as for most tuple of eigenvectors, this principle is indeed optimal. This is the topic of the forthcoming discussion.
 \medbreak

Let $\n\geq 1$, $d\geq 1$ and $\gamma^\n=(\gamma_1,...,\gamma_\n)$ be random elements in $ \S^{d-1}$ associated with the random covariance function
\begin{equation*}
	R_{\n}(t) = \frac{1}{\n}\sum_{k=1}^\n \cos{2\pi\langle \gamma_k,t\rangle}, \quad t\in \R^d.
\end{equation*}
\medbreak

We make the following assumptions:

\begin{assumption}\label{ass:gammaN}
~\medbreak
\begin{itemize}
	\item $\gamma$ is isotropic, i.e. for any rotation $\mathcal{R}$,
	\begin{equation*}
		\mathcal{R} \gamma^\n \stackrel{\mathcal{L}}{=} \gamma_\n.
	\end{equation*}
	\item For a $\mathcal C^\infty$ test function $h:\R^{d} \to [-1,1]$, Hoeffding's inequality is satisfied
	\begin{equation*}
		\mathbb{P}\left(\left|\frac{1}{\n}\sum_{k=1}^\n h(\gamma_k)-\int_{\S^{d-1}} h(\gamma) \sigma(\mathrm{d}\gamma) \right|>\varepsilon\right)\leq C\exp{-c\varepsilon^2 \n},
	\end{equation*}
	where $c,C$ are constants that do not depend on $h$ and $\sigma$ is the uniform measure on $\S^{d-1}$.
\end{itemize}

  \end{assumption}
The most natural model is that of i.i.d. random $\gamma _{i}$ uniformly distributed on $\S^{1}$. Now, if $\varepsilon_{\n}$ is a positive sequence converging to $0$, then we let $\widetilde{\tau}_{\n}$ be the smallest $\varepsilon_{\n}$-almost period associated with $R_{\n}$, that is the smallest $|\widetilde{\tau}_{\n}|$ such that
\begin{equation*}
	|\widetilde{\tau}_{\n}| \geq 1 \quad \text{ and } \quad R_{\n}( \widetilde\tau_{\n})>1-\varepsilon _{\n}.
\end{equation*}

According to Theorem \ref{theorem 2.1}, $\widetilde{\tau}_{\n}$ satisfies almost surely, for any $\alpha>0$, the inequality
\begin{equation}\label{equation 1}
	|\widetilde{\tau}_{\n}|\leq e^{\n^{1+\alpha}}
\end{equation}
 as long as $\n$ is large enough.\medbreak

 The next proposition ensures that this latter inequality is indeed optimal for a generic sequence of wave vectors $\gamma_{k}$.

\begin{proposition}\label{proposition 2.5}
	Under Assumption \ref{ass:gammaN} and for any $a< c/d$ where $c$ is the constant in Hoeffding's inequality,
	\begin{equation}\label{equation 2.5.1}
		\widetilde{\tau}_{\n} \geq e^{a\n}
	\end{equation}
	with high probability in the sense that for $\varepsilon_{\n} \to 0$
	\begin{equation*}
		\mathbb{P}\left(\sup_{1\leq |t| \leq e^{a\n}}R_{\n}(t)\leq 1-\varepsilon _{\n}\right)  \longrightarrow 1.
	\end{equation*}
\end{proposition}

In particular, the preceding theorem is satisfied when $\gamma^{\n}$ is a sequence of independent and identically uniform random variables on $\S^{1}$. Therefore, for a generic sequence of wave vectors, the bound obtained in Theorem \ref{theorem 2.1} cannot be decently improved. The proof is at Section \ref{sec:hoeffding}.

\subsection{Degrees of freedom and Dirichlet's approximation Theorem optimality}

Proposition \ref{proposition 2.5} implies that the bound derived from Dirichlet's approximation Theorem is optimal and that one cannot do better for a generic sequence of wave vectors. Yet, modelling the frequencies of the Arithmetic Random Wave as the realisation of some independent and identically distributed uniform random variable is not an accurate representation of reality as all the arithmetic flavour of the model is lost in the process. Following the idea of Sartori \cite{Sa2020}, it is possible to propose a more accurate  model. Recall that $n$ is expressible as the sum of two squares if
\begin{equation*}
	n = 2^{\alpha} \prod_{j=1}^{k}p_{j}^{\alpha_{j}}\prod_{j=1}^{l} q_{j}^{2\beta_{j}}
\end{equation*}
where $p_{i}$, $q_{k}$ are prime numbers satisfying respectively $p = 1 \pmod{4}$ and $q = 3 \pmod{4}$. The preceding prime decomposition induces a prime decomposition in the ring of Gaussian integers $\Z[i]$, namely
\begin{equation*}
	n=2^{\alpha} \prod_{j=1}^{k}\mathcal{P}_{j}^{\alpha_{j}}\overline{\mathcal{P}_{j}}^{\alpha_{j}}\prod_{j=1}^{l}q_{j}^{2\beta_{j}}
\end{equation*}
where $\mathcal{P}_{j}$ is one the two squares roots of $p_{j}$ in $\Z[i]$. Thus, if $n = \lambda_{1}^{2}+\lambda_{2}^{2}$, it follows that
\begin{equation*}
	\lambda_{1}+i\lambda_{2}=\nu Z^{\alpha} \prod_{j=1}^{k}\mathcal{P}_{j}^{\gamma_{j}}\overline{\mathcal{P}_{j}}^{\alpha_{j}-\gamma_{j}}\prod_{j=1}^{l}q_{j}^{\beta_{j}}
\end{equation*}
for some $0 \leq \gamma_{j} \leq \alpha_{j}$ and $\nu \in \{1,-1,i,-i\}$. Here, $Z$ refers to a square of $2$ in $\Z[i]$. \medbreak

For a generic integer $n$ expressible as the sum of two squares, 
most of the valuations $\alpha_{j}$ are equal to $1$  \cite{Sa2020}. We will make this assumption from now on. In that case, we denote by $\omega(n) = k$ the number of prime factors congruent to $3 \pmod{4}$ in the decomposition $n$ so that $2^{\omega(n)+2} = \mathcal{N}_{n}$.\medbreak

If we write
\begin{equation*}
	\theta_{j} := \arg(\mathcal{P}_{j}), \quad 1 \leq j \leq k,
\end{equation*}
and
\begin{equation*}
	\theta :=  \arg\left((1+\nu)^{\alpha}\prod_{j=1}^{l} q_{j}^{\beta_{j}}\right),
\end{equation*}

the correlation function of the Arithmetic Random Wave is then given by
\begin{equation}\label{equation 2}
	\widetilde{r}_{n}(t)=\frac{1}{\n_{n}}\sum_{\substack{\eta \in \{-1,1\}^{\omega(n)}\\ \nu \in \{0,1,2,3\}}} \cos{2\pi \left\langle\exp{i\frac{\pi}{2}\nu + i\sum_{j=1}^{\omega(n)} \eta_{j} \theta_{j}+i\theta},t \right\rangle}.
\end{equation}
In these settings, the wave vectors are dependent of the angles of the Gaussian primes. So, the system has only $\omega(n)$ degrees of freedom compared to $\n_{n}$ before. This potentially allows to decrease drastically the bound on almost periods from Dirichlet's Arithmetic principle for a system with $  \mathcal  N  $ degrees of freedom. We were not able to prove rigorously this idea but it is possible to hint on why Dirichlet's Arithmetic Principle should not be optimal in these settings.\medbreak

Consider the so-called \textit{linearised} covariance function, that is
\begin{equation}\label{equation 3}
	s_{n}(t) := \frac{4}{\n_{n}}\sum_{\eta \in \{-1,1\}^{\omega(n)}} \cos{2\pi \theta_{\eta}t }, \quad t\in \R^{2}
\end{equation}
where
\begin{equation}\label{equation 5}
	 \theta_{\eta}=\sum_{j=1}^{\omega(n)}\eta_{j} \theta_{j} \pmod{1}.
\end{equation}
Let us stress that this covariance function is deterministic, on the contrary of Sartori which considers a version with $\omega (n)$ random arguments  $\theta _{j} $.
In this model, the exponential terms in equation \eqref{equation 2} were linearised around the origin, so that the number of degrees of freedom of the model, which can be interpreted as $\omega (n)$, the numbers of signs to choose, is the same as in the ARW. The latter model provides an interesting example of a deterministic almost periodic function where the bound given by Dirichlet's approximation (Theorem \ref{theorem 2.2}) is no longer optimal.\medbreak

If $(\theta_{\eta})$ was the realisation of some independent and identically random process, we would expect the smallest $\varepsilon$-almost period to be of order $\varepsilon^{-\mathcal{N}_{n}}$, by Proposition \ref{proposition 2.5}. But, the family $(\theta_{\eta})$ is linearly dependent and using this observation, it is possible to improve drastically the upper bound.

\begin{proposition}\label{proposition 2.7}
	There is  $1\leqslant \tilde \tau _{n}\leqslant c\n_{n}^{\log(\log(\log(n)))}$ such that
	\begin{align*}
 		|s_{n}(\tilde \tau _{n})-1 | \to 0.
	\end{align*}
\end{proposition}

Hence this is a crude improvement over Dirichlet's Arithmetic Principle and closer in spirit to the scale $\n_{n}^{A}$ where full correlation occurs,  and Proposition \ref{proposition 2.5} applied with $\n= \omega (n)$ yields that it is optimal if the angles $\theta _{j}$ can be considered as  {\it generic} (i.e. behave like i.i.d.   variables). Yet, it is not possible to extend this result to covariance functions of the form of \eqref{equation 2} as the proof relies solely on linear considerations which are no longer satisfied outside the scope of this particular model. Yet, it shows that, for a model with $\omega (n)$ degrees of freedom, the bound from Dirichlet's approximation Theorem is no longer optimal. This leaves us with some unanswered but interesting questions.

\begin{question}\label{question 2.6}
	What is the size of the smallest sequence of almost periods in  model \eqref{equation 3} ? Is it the same as for the actual ARW?
	\end{question}

\subsection{Other manifolds and dimensions}
\label{sec:manifolds}
The existence of sequences of almost periods for the correlation function $r_{n}$ is directly related to the number of eigenvectors $\mathcal{N}_{n}$, and more precisely to the number of degrees of freedom of the system, this  is a core idea of our bounds. The more distinct frequencies an almost periodic function has, the greater the almost period will likely be. In  the case of the Arithmetic Random Wave, one has to remember that the norm of the almost period is constrained to live in the torus $\T^{2}$, or in the dilated torus $\sqrt{n}\T^{2}$ after rescaling. Luckily, the number of frequencies $\mathcal{N}_{n}$ grows logarithmically compared to $\sqrt{n}$. This is why it is  possible to find almost periods slightly above Planck scale. Let us give other examples where the dimension of the eigenspace is polynomial and where there is no full correlation or almost replication.

\subsubsection{Spherical harmonics}For the sphere $\S^{2}$, there is no possibility of nodal replication. In fact, the correlation function $r_{n}$ of the spherical harmonics of energy level $n(n+1)$ is related to the so-called Legendre Polynomials $P_{n}$ and is given by
\begin{equation*}
	r_{n}(x,y) = P_{n}(\langle x, y \rangle), \quad x,y \in \S^{1}.
\end{equation*}

The characteristic distance in this model is $1/n$ and Hilb's formula \cite[Theorem 8.21.6]{Sz1975}

\begin{equation*}
	P_{n}(\cos{\theta})=\sqrt{\frac{\theta}{\sin{\theta}}}J_{0}\left(\left(n+\frac{1}{2}\right)\theta\right)+O\left(n^{-3/2}\right), \quad 0 \leq \theta \leq \pi - \varepsilon.
\end{equation*}
ensures that $r_{n}$ converges uniformly to the Bessel function, except on small spherical caps located at the opposite pole. In particular, the latter formula prevents the existence of any almost period at a scale $1/n$ or slightly larger.

\subsubsection{Higher dimensional tori}      

The situation is similar if one considers the Arithmetic Random Wave on the $d$-dimensional torus $\T^{d}$. In that case, the eigenvalues consist in the energy levels $E_{n} =  -4\pi^{2}n$ where $n$ is a sum of $d$ squares and the associated eigenspace has dimension $\mathcal{N}_{n}^{(d)}$, that is the number of ways an integer $n$ can be written as sum of $d$ squares. For $d \geq 5$, the precise asymptotic were derived first by Hardy and Littlewood and proved by Hua \cite{Hua1938} in 1938:
\begin{equation*}
	\frac{\Gamma(3/2)^{d}}{\Gamma(d/2)} n^{d/2-1} S(n),
\end{equation*}
where $S(n)$ is the so-called singular series, bounded above and below. The leading term of the asymptotics is $n^{d/2-1}$ whereas for $d=2$, it was $\log(n)^{\log(2)/2}$. Hence, for large $d$, the number of eigenvalues does no longer grows slowly compared to $\sqrt{n}$. Dirichlet's approximation Theorem will thus no longer ensure the existence of a sequence of almost periods slightly above Planck scale. The authors of this paper are not aware of any positive or negative result related to full-correlation in dimension $d\geq 5$, but, if the preceding heuristics is true, one might expect that nodal replication as well as full correlation no longer hold for $d\geq 5$ as there are too many eigenvalues.

% !TeX spellcheck = en_GB

\section{Dirichlet's theorem for almost periodic fields}

The aim of the section is to prove the bounds on the smallest pseudo-period. To reach our goal, we will need to recall Dirichlet's approximation theorem. This is the content of the next paragraph.

\subsection{Dirichlet's approximation theorem}

 Dirichlet's approximation theorem is a standard tool in Diophantine approximation and quantifies how well a generic vector $\mu \in \R^{d}$ can approximate a vector of integers. In this paper, one will need a generalization of this principle, which deals with simultaneous approximations. Before stating the aforementioned theorem, we introduce the following notation. For $\mu \in \R^{d}$, let $\distZ{\mu}{}$ be the distance of $\mu$ to the nearest integer in $\Z$.\medbreak

 \begin{theorem}[Dirichlet's approximation theorem]\label{theorem 3.1}
 	Let $(\mu_{k})_{1 \leq k \leq \n}$ a sequence of elements in $\R^{d}$. For any integer $m>0$, there is $x \in \mathbb{Z}^{d}$ such that $1 \leq |x|_{\infty} \leq m^{\n/d}$ and
 	\begin{equation*}
 		\distZ{\langle \mu_{k},x \rangle}{} \leq \frac{1}{m}, \quad 1 \leq k \leq \n.
 	\end{equation*}
 	$|.|_{\infty}$ refers here to the supremum norm on $\R^{d}$.
 \end{theorem}

 The proof of this theorem relies on a clever use of the pigeonhole principle and is recalled briefly hereafter as it is quite interesting.

 \begin{proof}
 	Let $N = m^{\n/d}$ and consider $Q_{N}$ the set of lattice points $x \in \Z^{d}$ whose coordinates satisfy $1 \leq x_{k} \leq N$. Notice that $Q_{N}$ contains $N^{d} = m^{\n}$ points.\medbreak

 	There are now two possibilities. Either, there is some $x \in Q_{N}$ such that
 	\begin{equation*}
 		\distZ{\langle \mu_{k}, x \rangle}{} \leq \frac{1}{m}, \quad 1 \leq k \leq n.
 	\end{equation*}
 	In that case, there is nothing to prove .\medbreak

 	If not, split $[0,1]^{\n}$ in $m^{n}$ sub-squares of side length at most $\frac{1}{m}$ and associate with $x \in Q_{N}$ the vector $y_{x} \in \R^{\n}$ defined as
 	\begin{equation*}
 		y_{x} = (\langle \mu_{k},x \rangle \pmod{1})_{1 \leq k \leq \n}.
 	\end{equation*}
 	The set $\{y_{x} : x \in Q_{N}\}$ contains $N^{d} = m^{\n}$ elements and the above assumption ensures that none of these elements is contained in the lower sub-square $C = \left[0,\frac{1}{m}\right[^{\n}$. Hence, by the pigeonhole principle, they are two distinct elements $y_{x_{1}}$ and $y_{x_{2}}$ contained in the same square, that is
 	\begin{equation*}
 		\distZ{\langle \mu_{k}, y_{x_{1}}-y_{x_{2}} \rangle} \leq \frac{1}{m}, \quad 1 \leq k \leq \n
 	\end{equation*}
 	In that case, $x = y_{x_{2}}-y_{x_{1}}$ works.
 \end{proof}

As we have now proved Dirichlet's approximation theorem, it is time to concentrate on Theorem \ref{theorem 2.2}. In fact, we will prove a slight generalisation. Let $|.|$ be euclidean norm. It will be preferable to work with this norm instead of the uniform norm. This is not a limitation as all norms are equivalent on $\R^{d}$ and the conclusion of Dirichlet's approximation theorem remains the same, up to a constant we will neglect.

\begin{corollary}\label{corollary 3.2}
	Let $\n\geq 1$   and define the almost periodic function
	\begin{equation*}
		f(t) = \frac{1}{\n}\sum_{k=1}^{\n} \cos{2 \pi \langle \gamma_{k}, t\rangle}, \quad t\in \R^{d}
	\end{equation*}
	with $\gamma_{k}\in\R^{d}$. \medbreak

	For every $\varepsilon>0$, then there is an $\frac{1}{m}$-almost period $\tau$ satisfying
	\begin{equation*}
		1 \leq |\tau| \leq (2\pi)^{\n/d}m^{\n/d}
	\end{equation*}

	Moreover, if each $\gamma_{k}$ belong to $\S^{d-1}$, then $\tau$ is also an almost period for the derivatives of $f$ in the sense that for any $\alpha \in \N^{d}$, the following inequality is satisfied:
	\begin{equation*}
		\sup_{t \in \R^{d}} |\partial^{\alpha} f(t+\tau)-\partial^{\alpha}f(t)|\leq \frac{(2\pi)^{\alpha}}{m}
	\end{equation*}
	where $x^{\alpha} = x_{1}^{\alpha_{1}}...x_{d}^{\alpha_{d}}$.
\end{corollary}

\begin{proof}
	Dirichlet approximation theorem ensures the existence of $1 \leq |\tau|_{\infty} \leq m^{n/d}$ such that the following inequalities are all satisfied:
	\begin{equation*}
		\distZ{\langle \mu_{k},\tau \rangle} \leq \frac{1}{m}, \quad 1 \leq k \leq \n,
	\end{equation*}
	or equivalently
	\begin{equation*}
		\operatorname{dist}\left(2\pi \langle \mu_{k},\tau \rangle, 2\pi \Z\right) \leq \frac{2\pi}{m}, \quad 1 \leq k \leq \n.
	\end{equation*}
	Now, for any $t \in \R^{d}$:
	\begin{align*}
		|f(t+\tau)-f(t)| &= \frac{1}{\n}\left|\sum_{k=1}^{\n}\Re e^{2i\pi \langle \mu_{k}, t+\tau\rangle}-\Re e^{2i\pi \langle \mu_{k}, t\rangle}\right|\\
		&\leq \frac{1}{\n}\sum_{k=1}^{\n}\left|e^{2i\pi \langle \mu_{k}, \tau\rangle}-1\right|\\
		\text{(mean value theorem)}&\leq \frac{1}{n}\sum_{k=1}^{\n} \operatorname{dist}\left(2\pi \langle \mu_{k},\tau \rangle, 2\pi \Z\right)\\
		&\leq \frac{2\pi}{m}.
	\end{align*}
	If all the $\gamma_{k}$ are located on $\S^{1}$, then we can make use of the mean value inequality, noticing that
	\begin{align*}
		|\partial^{\alpha} f(t+\tau)- \partial^{\alpha} f(t)| &= \frac{(2\pi)^{\alpha}}{m}\left|\sum_{k=1}^{\n} \gamma_{k}^{\alpha} \left(\Re e^{2i\pi \langle \mu_{k}, t+\tau\rangle}-\Re e^{2i\pi \langle \mu_{k}, t\rangle}\right)\,\right|\\
		&\leq \frac{(2\pi)^{\alpha}}{m}\sum_{k=1}^{\n}|\gamma_{k}^{\alpha}|\left|e^{2i\pi \langle \mu_{k}, \tau\rangle}-1\right|
	\end{align*}
	The conclusion follows from the inequality $|\gamma_{k}^{\alpha}| \leq 1$.
\end{proof}

From now on, we will work with $\varepsilon>0$ instead of $\frac{1}{m}$ as the latter notation is more readable. It is not problematic as we are mainly interested into applying the latter result to very small $\varepsilon$. In that case, we can always find an integer $m>0$ such that $\varepsilon \approx \frac{1}{m}$. \medbreak

Now, using Corollary \ref{corollary 3.2}, we can deduce a slight improvement over Theorem \ref{theorem 2.1}. Not only Corollary \ref{corollary 3.2} implies the existence of sequence of almost periods $\tau_{n}$ associated with the covariance function of the Arithmetic Random Waves, but it also implies that $\tau_{n}$ is a sequence of almost periods for the derivatives of this function. This is a fundamental observation as derivatives of the covariance function control the derivatives of the underlying Gaussian field. This will be especially useful when dealing with the nodal lines. More precisely, we will need the content of the following proposition, which is an immediate consequence of the results above.

\begin{corollary}\label{proposition 3.3}
	If $\tau_{n}$ is a sequence of almost periods satisfying the assumptions of Theorem \ref{theorem 2.1}, then for any $\alpha \in \N^{2}$,
	\begin{equation*}
		\sup_{t \in \sqrt{n}\T^{2}} |\partial^{\alpha} \widetilde{r}_{n}(t+\tau_{n})-\partial^{\alpha}\widetilde{r}_{n}(t)|\longrightarrow 0.
	\end{equation*}
\end{corollary}

\subsection{When Dirichlet's approximation theorem fails}

This preceding theorem can be used to prove Proposition \ref{proposition 2.7}.

\begin{proof}Let $0<\varepsilon<1$.
	According to Theorem \ref{theorem 3.1} applied with $\n=\omega (n),d=1$, we can find some $|\tau| \leq \left(
\frac{ 2\pi \omega (n)}{\varepsilon }
\right)^{\omega(n)}$ such that
	\begin{equation*}
		\distZ{\theta_{i}\tau} \leq \frac{\varepsilon}{2\pi\omega(n)}, \quad 1\leq i \leq \omega(n).
	\end{equation*}
	For any $\eta \in \{-1,1\}^{\omega(n)}$,
	\begin{equation*}
		\distZ{\theta_{\eta}\tau} \leq \sum_{i=1}^{\omega (n)}\distZ{\theta_{i}\tau} \leq \frac{\varepsilon}{2\pi}.
	\end{equation*}
	Mirroring the proof of Corollary \ref{corollary 3.2}, we easily derive the inequality
	\begin{equation*}
		|s_{n}(t+\tau)-s_{n}(t)| \leq \varepsilon, \quad t \in \R.
	\end{equation*}
	To finish the proof, one can choose for instance $\varepsilon _{n}=\omega (n)^{-1}.$
\end{proof}

\section{Replication of the nodal lines}
\label{sec:lines}

In this section, the results related to the existence of almost periods are used to prove the phenomenon of the replication of the nodal lines as stated in Theorem \ref{proposition 2.3}. The proof   is based on a nice application of the continuous mapping theorem. But, before dwelling into the proof, we give the main ideas.\medbreak

At Planck scale, the Arithmetic Random Wave behaves almost exactly like a random eigenfunction of the Laplace operator on the plane $\R^{2}$ in the sense $\widetilde{T}_{n}$ that converges in distribution to a Berry Random Wave $T$ for the topology of uniform convergence on $\Omega$. On the other hand, we expect $\widetilde{T}_{n}'$ to behave very similarly to $\widetilde{T}_{n}$ in the high energy limit. It is thus natural to conjecture that $(\widetilde{T}_{n},\widetilde{T}_{n}')$ converges to $(T,T)$.\medbreak

Unfortunately, the present formulation is too weak to derive any meaningful result related to the nodal lines of the Arithmetic Random Wave. What will be needed is a control on the behaviour of the derivatives of the underlying field and that
 control will be offered by Corollary \ref{proposition 3.3}. Therefore, instead of proving that $(\widetilde{T}_{n},\widetilde{T}_{n}')$ converges to $(T,T)$ for the topology of uniform convergence, we will prove that the convergence holds also for the topology of uniform convergence of the first $p$ derivatives.

\subsubsection*{Notation} In what follows, $\Omega \subset \R^{2}$ refers to a convex, bounded, open set with smooth boundary and $\mathcal{C}^{p}(\overline{\Omega})$ refers to the set of $p$ times differentiable mappings on $\Omega$ whose partial derivatives extends continuously to the boundary. This space is endowed with the norm
\begin{equation*}
	||f||_{p} = \sum_{|\alpha| \leq p}\sup_{\Omega} |\partial^{\alpha} f|
\end{equation*}
where $|\alpha| =\alpha_{1}+\alpha_{2}$.  This norm turns $\mathcal{C}^{p}(\overline{\Omega})$ into a Polish space. Further details are recalled in the Appendix.

 \subsection{Proof of the replication phenomenon}

Let
\begin{equation*}
	U_{n}(s,t) := (\widetilde{T}_{n}(s), \widetilde{T}_{n}'(t)), \quad s,t \in \overline{\Omega},
\end{equation*}
and
\begin{equation*}
	U(s,t) := (T(s),T(t)),    \quad s,t \in \overline{\Omega}.
\end{equation*}

Both $U_{n}$ and $U$ define processes on the product space $\mathcal{C}^{p}(\overline{\Omega})^{2}$ endowed with the product topology. In particular, $\mathcal{C}^{p}(\overline{\Omega})^{2}$ is a Polish space.

\begin{proposition}\label{proposition 4.1}
	Under the preceding assumptions, $(U_{n})$ converges in distribution to $U$.
\end{proposition}

\begin{proof}
	As $\mathcal{C}^{p}(\overline{\Omega})^{2}$ is Polish space, convergence in distribution is equivalent to tightness and unicity of the limit.\medbreak

	\textit{Step 1: tightness}\medbreak

	The sequence $(U_{n})$ is tight if and only if its marginals $(\widetilde{T}_{n})$ and $(\widetilde{T}_{n}')$ are tight. As the Arithmetic Random Wave is a stationary process, $\widetilde{T}_{n}$ and $\widetilde{T}_{n}'$ have the same distribution. Hence, we only have to prove tightness for $\widetilde{T}_{n}$ in $\mathcal{C}^{p}(\overline{\Omega})$. Following the conclusion of Proposition \ref{lemma 6.3}, this amounts to prove that $(\partial^{\alpha}\widetilde{T}_{n})$ is tight in $\mathcal{C}^{0}(\overline{\Omega})$ for any $|\alpha| \leq p$. This problem is very tractable as tightness in $\mathcal{C}^{0}(\overline{\Omega})$ is a deeply studied topic.\medbreak

	We will use Kolmogorov tightness criterion \cite[Corollary 16.9]{Kal2002}. $\partial^{\alpha}\widetilde{T}_{n}$ is a stationary centered Gaussian process with correlation function
	\begin{equation*}
		\widetilde{r}_{\alpha,n}(t) := (-1)^{|\alpha|}\partial^{2\alpha}\widetilde{r}_{n}(t), \quad t\in \R^{2}.
	\end{equation*}

	Fixed for a moment $t \in \Omega$. $\partial^{\alpha}\widetilde{T}_{n}(t)$ is  then a centered Gaussian random variable whose variance $\sigma_{n}^{2}$ is converging to
	\begin{equation*}
		\sigma^{2} := (2\pi)^{2|\alpha|}\int_{\S^{1}}\lambda^{2\alpha} \sigma(\mathrm{d}\lambda)>0.
	\end{equation*}
	where $|\alpha| =\alpha_{1}+\alpha_{2}$ and $\sigma$ is the uniform probability on $\S^{1}$. Hence, $\left(\partial^{\alpha}\widetilde{T}_{n}(t) \right)$ converge in distribution to a centered normal random variable with variance $\sigma^{2}$.\medbreak

	In the meantime, $\partial^{\alpha}\widetilde{T}_{n}(t)-\partial^{\alpha}\widetilde{T}_{n}(s)$ is also a centered Gaussian random variable with variance
	\begin{equation*}
		\sigma^{2}(t-s) := 2(\widetilde{r}_{\alpha,n}(0)-\widetilde{r}_{\alpha,n}(t-s))
	\end{equation*}

	Combining the triangle inequality with
	\begin{equation*}
		1-\cos{t} \leq \frac{t^{2}}{2}, \quad t\in \R^{2},
	\end{equation*}
	one has
	\begin{align*}
		|\widetilde{r}_{\alpha,n}(0)-\widetilde{r}_{\alpha,n}(t-s)| &\leq \frac{1}{\n_{n}}\sum_{\lambda \in \Lambda_{n}}\left|1-\cos{2\pi\left\langle \frac{\lambda}{\sqrt{n}},t-s\right\rangle}\right|\\
		&\leq \frac{1}{2\n_{n}}\sum_{\lambda \in \Lambda_{n}} \left\langle \frac{\lambda}{\sqrt{n}},t-s\right\rangle^{2}\\
		\text{(Cauchy Schwarz)}&\leq \frac{1}{2}|t-s|^{2}.
	\end{align*}

	In particular,

	\begin{align*}
		\esp{\left(\partial^{\alpha}\widetilde{T}_{n}(t)-\partial^{\alpha}\widetilde{T}_{n}(s)\right)^{4}} &= 3 \sigma(t-s)^{4}\\
		&\leq 3|t-s|^{4}.
	\end{align*}

	The conditions of Kolmogorov criterion are met. Hence, the sequence $(\partial^{\alpha}\widetilde{T}_{n})$ is tight in $\mathcal{C}^{0}(\overline{\Omega})$.\medbreak

	\textit{Step 2: Unicity of the limit}\medbreak

	The distribution of $\mathcal{C}^{p}(\overline{\Omega})^{2}$ is characterised by its finite dimensional distributions (see Lemma \ref{lemma 6.4} and the associated discussion). It suffices thus to prove that
	\begin{equation*}
		U_{n}^{(p)} = (U_{n}(t_{1}), ..., U_{n}(t_{p}))
	\end{equation*}
	converges in distribution to
	\begin{equation*}
		U^{(p)} = (U(t_{1}), ..., U(t_{p})),
	\end{equation*}
	where $t_{i}$ is a finite sequence of points in $\Omega$.\medbreak

	As both $U_{n}^{(p)}$ and $U^{(p)}$ are centered Gaussian processes, it is only needed to show that the covariance $U_{n}^{(p)}$ converges to the covariance of $U^{(p)}$. In terms of covariance function, this amounts to prove that
	\begin{equation*}
		\left\{
		\begin{aligned}
			&\widetilde{r}_{n}(t_{j}-t_{i}) \longrightarrow \widetilde{r}(t_{i}-t_{j}), \quad 1 \leq i,j \leq n.\\
			&\widetilde{r}_{n}'(t_{j}-t_{i}) \longrightarrow  \widetilde{r}(t_{i}-t_{j})
		\end{aligned}
		\right.
	\end{equation*}
	The latter statement is trivial as $\widetilde{r}_{n}$ converges uniformly to $\widetilde{r}$ (Lemma \ref{lemma 1.3}) and $\widetilde{r}_{n}-\widetilde{r}_{n}'$ converges uniformly to $0$ (Theorem \ref{theorem 2.1}). This finishes the proof.
\end{proof}

In order prove nodal replication as introduced by Theorem \ref{proposition 2.3}, it is natural to consider the application
\begin{equation*}
	\Gamma : \left\{
	\begin{array}{ccc}
		\mathcal{C}^{1}\left(\overline{\Omega}, \mathbb{R}\right) &\longrightarrow& \mathbb{R} \cup \{\pm \infty\}\\
		f &\longmapsto& \displaystyle{\int_{\mathcal{Z}_{f}(\Omega)} \varphi(t) H_{1}(\mathrm{d}t)}.
	\end{array}
	\right.
\end{equation*}
where $\mathcal{Z}_{f}(\Omega)$ is the zero set of $f$ on $\Omega$ and $\varphi :\Omega \to \R$ is continuous. Instead of considering functions $f \in \mathcal{C}^{p}(\overline{\Omega})$, it will preferable to work with a slightly larger domain $U$ such that $\overline{\Omega} \subset U$ as this will allow some more leeway.\medbreak

First, the co-area formula
\begin{equation*}
	\Gamma(f) = \lim_{\varepsilon \to 0}\frac{1}{2\varepsilon}\int_{\Omega} \mathds{1}(|f(t)| \leq \varepsilon) \varphi(t)|\nabla f(t)| \mathrm{d}t
\end{equation*}
ensures that $\Gamma$ is a measurable mapping. This allows to define $\Gamma(\widetilde{T}_{n})$ and Theorem \ref{proposition 2.3} amounts to show that $\Gamma(\widetilde{T}_{n})-\Gamma(\widetilde{T}_{n}')$ converges in distribution to $0$. The latter difference could be ill-defined if both $\Gamma(\widetilde{T}_{n})$ and $\Gamma(\widetilde{T}_{n}')$ happens to be infinite at the same time, but this situation is excluded by Bulinskaya's lemma \cite[Proposition 1.20]{AW2009} which ensures that almost surely $\widetilde{T}_{n}$ is regular.\medbreak

\begin{definition}\label{definition 4.2}
	$f \in \mathcal{C}^{1}(\overline{U})$ is regular if and only
	\begin{equation*}
		\forall t \in \overline{U}, f(t)=0 \implies \nabla f(t) \neq 0.
	\end{equation*}
\end{definition}

Regular functions have the property that their zero sets are $1$ dimensional smooth manifold. This ensures in particular that $\Gamma(f)$ is not infinite and that the random variable $\Gamma(\widetilde{T}_{n})-\Gamma(\widetilde{T}'_{n})$ is well-defined. Theorem \ref{proposition 2.3} will follow easily easily from the continuous mapping theorem \cite[Theorem 4.27]{Kal2002} if we manage to prove that $\Gamma$ is $\mathbb{P}_{T}$ almost-surely continuous. This result is not immediate as $\Gamma$ is not continuous everywhere.

\begin{example}\label{example 4.3}
	Let $\Omega$ be the open unit ball on $\R^{2}$. let $f_{\alpha}(t)=\alpha-|t|^{2}$. It is easily checked that $\Gamma(f_{1}) = 0$. Yet, $\Gamma(f_{\alpha}) = 2\pi \alpha$ for $\alpha \to 1^{-}$ even though $f_{\alpha} \to f_{1}$ as $\alpha$ tends to $1^{-}$.
\end{example}

In the preceding example, the zeros concentrate alongside the boundary and that is the root cause of the default of continuity. In our random settings, such situation is very unlikely to happen as it is testified by the following proposition.

\begin{proposition}\label{proposition 4.4}
	\begin{equation*}
		H_{1}\left(\mathcal{Z}_{T}(\partial\Omega)\right) <+\infty, \quad a.s.
	\end{equation*}
\end{proposition}

\begin{proof}
	 Denote by $H_{0}$ the counting measure. We will prove the stronger statement that the number of intersections between $\mathcal{Z}_{f}( \partial \Omega)$ is finite, that is $H_{0}(\mathcal{Z}_{f}( \partial \Omega)) < +\infty$ almost-surely. \medbreak

	 As $\partial \Omega$ is compact, it suffices to show that the latter result holds locally. Fix $\omega \in \partial\Omega$ and parametrise $\partial \Omega$ in a neighbourhood of $\omega$ by a smooth map $\gamma : \overline{V} \to \partial \Omega$ where $V$ is a bounded neighbourhood of $0$. We have to show that the Gaussian process $T \circ \gamma$ has only a finite number of zeros in $\overline{V}$. But, this follows immediately from Kac-Rice formula \cite[Theorem 6.2]{AW2009}.
\end{proof}

Excluding the latter pathological edge cases, $\Gamma$ is continuous. More precisely,

\begin{proposition}\label{proposition 4.5}
	Let $f$ be a regular function, $\varphi$ continuous, such that
	\begin{equation}\label{equation 4.5.1}
		H_{1}(\mathcal{Z}_{f}(\partial \Omega))=0,
	\end{equation}
	then $\Gamma$ is continuous in $f$.
\end{proposition}

The proof of this result is technical and not very enlightening. It is thus proved in the Appendix (see Proposition \ref{proposition 6.6}). Using the latter result, we deduce in particular that $\Gamma$ is $\mathbb{P}_{T}$ surely continuous and Theorem \ref{proposition 2.3} follows easily for the continuous mapping theorem.\medbreak

Before concluding this section, we can make a few comments. The requirement of taking a convex open set $\Omega$ with smooth boundary can be probably relaxed. Secondly, the result still holds, if instead of taking a fixed open set $\Omega$, we take a sequence of slowing increasing open set $\Omega_{n}$. The preceding proof no longer works in that case but it is possible to circumvent the issue with a careful use of Borel-Tis inequality or by some coupling arguments.

\section{Optimality of Dirichlet's approximation theorem}
\label{sec:hoeffding}

The goal of this section is to investigate the optimality of Dirichlet's approximation theorem and prove Proposition \ref{proposition 2.5}.\medbreak

In what follows, $\gamma^{\n}$ is a sequence of wave vectors satisfying Assumption \ref{ass:gammaN}. In particular, $\gamma^{\n}$ can be a sequence of independent uniformly distributed on $\S^{d-1}$. We start with a preliminary and immediate lemma.

\begin{lemma}\label{lemma 5.1}
	Under the preceding assumptions, each component of $\gamma^{\n}$ follows a uniform distribution on the sphere $\S^{d-1}$ and we let for $t \in \R^{d}$
	\begin{equation*}
		R(t)=\esp{R_{\n}(t)}=\int_{\S^{d-1}}\cos{2\pi \langle \gamma, t \rangle} \sigma(\mathrm{d}\gamma),
	\end{equation*}
	where $\sigma$ is the uniform measure on the sphere.\medbreak

	Then, $R$ defines a radial function which converges to $0$ at infinity and which has a unique maximum at $t=0$.
\end{lemma}

\begin{proof}
	 It is standard fact that the only rotation invariant probability on the sphere $\S^{d-1}$ is the uniform measure Hausdorff measure on $\S^{d-1}$. Hence, each component of $\gamma^{\n}$ has to follow a uniform distribution.\medbreak

	 We notice that, by symmetry,
	 \begin{equation*}
	 	R(t)=\int_{\S^{d-1}}e^{2i\pi \langle \gamma,t\rangle} \sigma(d\gamma)
	 \end{equation*}
 	so that $R(t) = \widehat{\sigma}(t)$ where $\widehat{.}$ refers to the Fourier transform. In particular, this leads to an explicit formula for $R(t)$, that is
 	\begin{equation*}
 		R(t)=\omega\frac{J_{\nu}(2\pi|t|)}{|t|^{\nu}}
 	\end{equation*}
 	where $J_{\nu}$ is the $\nu$\textsuperscript{th} Bessel function of first kind,  $\nu = d/2-1$ and $\omega$ a constant chosen so that $R(0)=1$. The conclusion of Lemma \ref{lemma 5.1} follows then from the standard properties of the Bessel functions.
\end{proof}

We are now ready to prove Proposition \ref{proposition 2.5}. Let $a<c/d$ where $c$ is the constant in Hoeffding's inequality and let $\varepsilon_{\n}$ be a positive sequence converging to $0$. We define $\tau_{\n}$ as the smallest $\varepsilon_{\n}$-almost period of $R_{\n}$. According to Hoeffding's inequality in Assumption \ref{ass:gammaN},
\begin{equation}\label{equation 4}
	\mathbb{P}\left(\left|R_{\n}(t)-R(t)\right|\geq 1-\varepsilon_{\n}\right) \leq Ce^{-c(1-\varepsilon_{\n})^{2}\n}, \quad t \in \R^{d},
\end{equation}

To make use of the latter inequality, we cover the centered ball with radius $e^{a\n}$ with $d_{\n}=O\left(\n^{d} e^{da\n}\right)$ balls of radius at most $\frac{1}{\n}$. Let $C_{i,\n}$ denote these latter balls and $t_{i,\n}$ be their respective centers. As both $R_{\n}$ and $R$ are $2\pi$-Lipschitz continuous, one has that
\begin{equation*}
	\sup_{t \in C_{i,n}} |R_{\n}(t)-R(t))| \leq |R_{\n}(t_{i,\n})-R(t_{i,\n})|+\frac{2\pi}{\n}.
\end{equation*}

Hence, for any $\varepsilon>0$,
\begin{equation*}\small
	\mathbb{P}\left(\limsup_{\n\to+\infty}\left\{\sup_{|t|\leq e^{a\n}} |R_{\n}(t)-R(t)| \geq 1-\varepsilon_{\n}\right\}\right)=\mathbb{P}\left(\limsup_{\n\to+\infty}\left\{\sup_{1 \leq i \leq d_{\n}} |R_{\n}(t_{i,\n})-R(t_{i,\n})| \geq 1-\varepsilon_{\n}\right\}\right).
\end{equation*}
By the union bound and Equation \eqref{equation 4},
\begin{align*}
	\mathbb{P}\left(\sup_{1 \leq i \leq d_{\n}}|R_{\n}(t_{i,\n})-R(t_{i,\n})| \geq 1-\varepsilon_{\n}\right) &\leq d_{n} Ce^{-c(1-\varepsilon_{\n})^{2}\n}\\
	&=O\left(\n^{d}e^{-c(1-\varepsilon_{\n})^{2} \n+ ad\n}\right)\\
	&=O\left(\n^{d}e^{-(c-ad)\n}\right)
\end{align*}
which is summable as $\varepsilon_{\n}$ converges to $0$. Borel-Cantelli lemma implies that
\begin{equation*}
	\mathbb{P}\left(\limsup_{n\to+\infty}\left\{\sup_{|t|\leq e^{a\n}} |R_{n}(t)-R(t)| \geq  1-\varepsilon_{\n}\right\}\right) = 0.
\end{equation*}
Hence, $R_{\n}$ converges uniformly to $R$ on balls of radius $e^{a\n}$. But, as $R(t)$ is bounded away from $1$ everywhere except on a neighbourhood of $0$, this implies that  there is asymptotically no $\varepsilon_{\n}$-almost periods smaller than $e^{a\n}$. In particular, Dirichlet's approximation theorem bound is the best possible in that case and Proposition \ref{proposition 2.7} is proved.

\section{Appendix}

\subsection{Topology of uniform convergence of derivatives}

This first chapter recalls briefly the main results associated with the topology of uniform convergence of the first $p$ derivatives on a compact set. In what follows, $\Omega \subset \R^{2}$ will designate a convex, bounded open set with locally $\mathcal{C}^{1}$ boundary. These latter assumptions can be relaxed, but in the framework of this article, there is no need to dwell into greater generality.

\subsubsection*{Notation}
When $\alpha \in \N^{2}$, the notation $|\alpha| = \alpha_{1}+\alpha_{2}$ will be used. In other context, it will refer to the standard euclidean norm.

\subsubsection*{Topology of uniform convergence of derivatives}

Let $p$ be a positive integer. We define the set $\mathcal{C}^{p}(\overline{\Omega})$ of $p$-times differentiable functions on $\overline{\Omega}$ as the set of functions $f : \overline{\Omega} \to \R$ satisfying
\begin{enumerate}
	\item $f$ is $\mathcal{C}^{p}(\Omega)$, that is $f$ is $p$-times continuously differentiable on $\Omega$.
	\item  $\partial^{\alpha}f$ extend continuously to $\overline{\Omega}$ for any $|\alpha| \leq p$.
\end{enumerate}

It is natural to equip $\mathcal{C}^{p}(\overline{\Omega})$ with the norm

\begin{equation*}
	||f||_{p}:=\sum_{|\alpha| \leq p} ||\partial^{\alpha} f||_{\infty}
\end{equation*}
where $||.||_{\infty}$ is the uniform norm on $\overline{\Omega}$.\medbreak

In particular, the definition of $\left(\mathcal{C}^{0}(\overline{\Omega}), ||.||_{0}\right)$ coincides with the topology of uniform convergence on compact set for continuous function. We recall now some standard facts on $\mathcal{C}^{p}(\overline{\Omega})$.

\begin{lemma}\label{lemma 6.1}
	$\mathcal{C}^{p}(\overline{\Omega})$ endowed with $||.||_{p}$, is a Polish space, that is separable and complete.
\end{lemma}

\begin{proof}
	The family of polynomials $\Q[x,y]$ is a dense family in $\mathcal{C}^{p}(\overline{\Omega})$. This result is surely true, but it is not as obvious as it seems... Indeed, for pathological bounded open sets $\Omega$, the density of polynomials may dramatically fail (see this discussion \cite{MO}).\medbreak

	In our setting, the regularity assumptions put on $\Omega$ forbids any pathological behaviour. As $\Omega$ is convex, it satisfies the assumptions of Whitney extension theorem \cite{Wh1934}, which ensures that any $f \in \mathcal{C}^{p}(\overline{\Omega})$ can be extended to a function of class $\mathcal{C}^{p}(\R^{2})$.\medbreak

	The separability follows then from a routine argument. Consider a mollifier, for instance the density of standard $2$-dimensional random normal variable
	and let $\varphi_{n}(t):=2^{-n}\varphi(2^{n}t)$. Then $\varphi_{n} \circledast f$ is a smooth function whose all partial derivatives of order $|\alpha| \leq p$ converge uniformly on compact sets to those of $f$. Weierstrass theorem ensures that $\varphi_{n}$ can be well-approximated by a polynomial $P_{n}$ on the euclidean unit ball $B$ and we might choose the sequence $P_{n}$ so that
	\begin{equation*}
		\sup_{B}|P_{n}-\varphi_{n}|\to 0.
	\end{equation*}
	To conclude the proof, one has to check that
	\begin{equation*}
		\mathds{1}_{B}P_{n}\circledast f
	\end{equation*}
	is a polynomial which approximates arbitrarily closely $f$ in the $\mathcal{C}^{p}(\overline{\Omega})$ topology.

	We now prove the completeness of $\mathcal{C}^{p}(\overline{\Omega})$. Recall that $\mathcal{C}^{0}(\overline{\Omega})$ equipped with $||.||_{\infty}$ is a Banach space.\medbreak

	Let $(f_{n})$ be a Cauchy sequence in $\mathcal{C}^{p}(\overline{\Omega})$. For every $|\alpha| \leq p$, $(\partial^{\alpha} f_{n})$ is a Cauchy sequence in $\mathcal{C}^{0}(\overline{\Omega})$, hence there is a continuous function $f_{\alpha}$ on $\overline{\Omega}$ such that

	\begin{equation*}
		||\partial^{\alpha}f_{n} - f_{\alpha}||_{\infty}  \longrightarrow 0.
	\end{equation*}

	We note $f:=f_{(0,0)}$. To finish the proof, it suffices to show that $f$ is $\mathcal{C}^{p}(\Omega)$ and $\partial^{\alpha}f=f_{\alpha}$. This follows immediately from the following well-known lemma.

	\begin{lemma}\label{lemma 6.2}
		Let $(f_{n})$ be a sequence of differentiable functions such that $f_{n}$ converges pointwise to $f$ and $f_{n}'$ converges uniformly to $g$. Then, $f$ is differentiable on $]a,b[$ and $f'=g$.
	\end{lemma}
\end{proof}

Results related to the compactness of $\mathcal{C}^{0}(\overline{\Omega})$ such as Arzelà–Ascoli have a straightforward extension to $\mathcal{C}^{p}(\overline{\Omega})$. This derives from the following observation.

\begin{lemma}\label{lemma 6.3}
	Let $A \subset \mathcal{C}^{p}(\overline{\Omega})$. The two following proposition are equivalent.

	\begin{enumerate}
		\item $A$ is relatively compact in $\mathcal{C}^{p}(\overline{\Omega})$.
		\item For any $|\alpha | \leq p$,
		\begin{equation*}
			A(\alpha) : = \{\partial^{\alpha} f :  f\in A\}
		\end{equation*}
		is relatively compact in $\mathcal{C}^{0}(\overline{\Omega})$
	\end{enumerate}
\end{lemma}

\begin{proof}
	The direct implication is immediate as
	\begin{equation*}
		||\partial^{\alpha}f||_{\infty} \leq ||f||_{p}.
	\end{equation*}

	For the converse, consider a subsequence $f_{n}$ in $A$. Using the relative compactness of $A(\alpha)$ in $\mathcal{C}^{0}(\overline{\Omega})$, we can extract a subsequence $f_{\sigma(n)}$ such that
	\begin{equation*}
		||\partial^{\alpha}f_{\sigma(n)} -f_{\alpha}||_{\infty} \longrightarrow 0
	\end{equation*}
	where each $f_{\alpha}$ is some continuous function on $\overline{\Omega}$ and similarly to the proof of Lemma \ref{lemma 6.1}, one has $\partial^{\alpha} f = f_{\alpha}$ where $f:=f_{(0,0)}$.
\end{proof}

\subsubsection*{Borel sets and convergence in distribution in $\mathcal{C}^{p}(\overline{\Omega})$}

In the previous paragraph, it was shown that $\mathcal{C}^{p}(\overline{\Omega})$ is Polish space. This is the natural framework to develop a well-behaved notion of convergence in distribution as Prokhorov's theorem is satisfied. The treatment of this question is inspired by Kallenberg \cite[Chapter 15]{Kal2002}.\medbreak

There are two natural notions of measurable sets in $\mathcal{C}^{p}(\overline{\Omega})$, one is given by the Borel-algebra $\borel{p}$, the other is given by the cylindrical $\sigma$-algebra $\Cyl{p}$, that is the coarsest  $\sigma$-algebra making the projection
\begin{equation*}
	\Pi_{I} :
	\left\{
	\begin{array}{lllc}
		&\mathcal{C}^{p}(\overline{\Omega}) &\longrightarrow &\R^{I}\\
		&f &\longmapsto  &\{f(t) : t \in I\}
	\end{array}
	\right.
\end{equation*}
measurable where $I$ is any finite subset of $\Omega$. Luckily, these two approaches can be unified into a unique one as they lead to the same measurable sets.

\begin{lemma}\label{lemma 6.4}
	$\borel{p} = \Cyl{p}$
\end{lemma}

\begin{proof}
	$\Pi_{I}$ is continuous, hence $\Cyl{p} \subset \borel{p}$.\medbreak

	For the converse, it suffices to show that any open set belongs to $\Cyl{p}$. As $\mathcal{C}^{p}(\overline{\Omega})$ is separable, it further reduces into showing that any closed ball $B$ is  $\Cyl{p}$-measurable. We denote by $f$ the center of $B$ and $r$ its radius.\medbreak

	If $A := \{\alpha \in \N^{2} : |\alpha| \leq p\}$ and $D$ is a dense countable subset of $\Omega$, then
	\begin{equation*}
		B = \bigcap_{(t_{\alpha}) \in D^{A}} \left(g \in \mathcal{C}^{p}(\overline{\Omega}) : \sum_{|\alpha| \leq p} \left|\partial^{\alpha} g(t_{\alpha})-\partial^\alpha f(t_{\alpha})\right| \leq r \right).
	\end{equation*}

	Thus, proving the reverse inclusion amounts to prove that
	\begin{equation*}
		\Pi^{\alpha}_{t} :
		\left\{
		\begin{array}{lllc}
			&\mathcal{C}^{p}(\overline{\Omega}) &\longrightarrow &\R\\
			&f &\longmapsto  &\partial^{\alpha} f(t)
		\end{array}
		\right.
	\end{equation*}
	is  $\Cyl{p}$-measurable.\medbreak

	Let $(e_{1}, e_{2})$ be the canonical basis of $\R^{2}$ and let
	\begin{equation*}
		\Delta_{h}^{i}f(t) := \frac{f(t+he_{i})-f(t)}{h}
	\end{equation*}
	be the discrete partial derivatives of order $f$. This definition extends naturally to higher derivatives $\Delta_{h}^{\alpha}$ and Taylor formula ensures that
	\begin{equation*}
		\lim_{h \to 0} \Delta_{h}^{\alpha}f(t) = \partial^{\alpha}f(t), \quad t \in \Omega.
	\end{equation*}
	As a pointwise limit of measurable functions, $\Pi_{t}^{\alpha}$ is thus $\Cyl{p}$-measurable.
\end{proof}

In particular, the preceding lemma ensure that a distribution in $\mathcal{C}^{p}(\overline{\Omega})$ is characterised by its finite-dimensional distributions.\medbreak

More generally, it ensures that a random variable defined on the product space $\mathcal{C}^{p}(\overline{\Omega})^{2}$ is also characterised by its finite dimensional. Indeed,
as $\mathcal{C}^{p}(\overline{\Omega})$ is separable, one has
\begin{align*}
	\mathcal{B}\left(\mathcal{C}^{p}(\overline{\Omega})^{2}\right)&=\borel{p}\otimes\borel{p}\\
	&=\Cyl{p}\otimes \Cyl{p}
\end{align*}
This result was used in the proof of Proposition \ref{proposition 4.1}.\medbreak

Alongside the cylindrical topology, we will need some results relative to the tightness of a sequence of random variables in $\mathcal{C}^{p}(\overline{\Omega)}$.

\begin{proposition}\label{proposition 6.5}
	Let $(X_{n})$ a sequence random variables on $\mathcal{C}^{p}(\overline{\Omega})$. Then, the two proposition are equivalent:
	\begin{enumerate}
		\item $(X_{n})$ is tight.
		\item For any $|\alpha| \leq p$, the sequence $(\partial^{\alpha}X_{n})$ is tight for the topology of uniform convergence on $\overline{\Omega}$.
	\end{enumerate}
\end{proposition}

\begin{proof}
	The proof relies on Prokhorov theorem \cite[Theorem 16.3]{Kal2002}. In a Polish space, a sequence of random variable is tight if and only if it is relatively compact.\medbreak

	But according to Lemma \ref{lemma 6.3}, $(X_{n})$ is tight in $\mathcal{C}^{p}(\overline{\Omega})$ if and only if $(\partial^{\alpha}X_{n})$ is tight in  $\mathcal{C}^{0}(\overline{\Omega})$ for any $|\alpha|\leq p$.
\end{proof}

\subsection{Continuity along the nodal lines}

The second part of this appendix is dedicated to the proof of Proposition \ref{proposition 4.5}. This work is a bit tedious and it is divided into smaller steps.

\subsubsection*{Notation}

If $g$ is continuous mapping, the notation $\mathcal{Z}_{g}(\Omega)$ will be used to denote the zero sets of $g$ on $\Omega$.

\subsubsection*{General framework}We recall first Proposition \ref{proposition 4.5}.\medbreak

Let $U \subset \R^{2}$ be a bounded convex open set with smooth boundary and $f \in \mathcal{C}^{1}(\overline{U})$ be a regular function. If $\Omega$ is an open convex subset of $U$ and $\varphi$ a continuous function on $\Omega$, then, provided that $f^{-1}(0)$ does not accumulate on $\partial \Omega$, the map
\begin{equation*}
	\Lambda(g):=\int_{\mathcal{Z}_{g}(\Omega)}f(t)H_{1}(\mathrm{d}t)
\end{equation*}
is continuous at $f$. More precisely, the goal is to prove the following proposition.

\begin{proposition}\label{proposition 6.6}
	 Under the preceding assumptions and provided that
	\begin{equation}\label{equation 6.6.1}
		H_{1}(\mathcal{Z}_{f}(\partial \Omega))=0,
	\end{equation}
	the map $\Gamma$ is continuous in $f$ for the $\mathcal{C}^{1}(\overline{U})$-topology.
\end{proposition}

The latter result can be extended to open sets $\Omega$ which are well-approximated above and below in the sense that there exists $\Omega_{n}$ and $\Omega^{n}$ two sequences of respectfully increasing and decreasing open sets such that
\begin{equation*}
	\bigcup_{n \geq 1} \Omega_{n} = \Omega
\end{equation*}
and
\begin{equation*}
	\bigcap_{n \geq 1} \Omega^{n} = \overline{\Omega},
\end{equation*}
provided some regularity on $\Omega$. In the case of convex sets, these conditions are immediately satisfied.

\subsubsection*{Proof of continuity} By decomposing $\varphi$ into its positive and negative parts, one has to show that Proposition \ref{proposition 6.6} holds whenever $\varphi$ is positive. Under this additional assumption, the latter proposition is induced by the following result.

\begin{proposition}\label{proposition 6.7}
	Let $C \subset D$ be any convex open sets included in $U$ such that $\overline{C} \subset D$. Then
	\begin{equation*}
		\limsup_{g \to f} \int_{\mathcal{Z}_{g}(C)} \varphi(t)H_{1}(\mathrm{d}t) \leq \int_{\mathcal{Z}_{f}(D)}\varphi(t) H_{1}(\mathrm{d}t)
	\end{equation*}
	and
	\begin{equation*}
		\int_{\mathcal{Z}_{f}(C)}\varphi(t) H_{1}(\mathrm{d}t) \leq \liminf_{g \to f} \int_{\mathcal{Z}_{g}(D)} \varphi(t)H_{1}(\mathrm{d}t).
	\end{equation*}
\end{proposition}

Before proving this result, we will show how Proposition \ref{proposition 6.7} implies the continuity of $\Lambda$ at $f$.

\begin{proof}[Proof of Proposition \ref{proposition 6.6}]
	Let $\Omega_{n}$ be a sequence of increasing convex sets such that $\overline{\Omega_{n}} \subset \Omega_{n+1}$ and
	\begin{equation*}
		\bigcup_{n \geq 1}\Omega_{n} = \Omega.
	\end{equation*}
	From Proposition \ref{proposition 6.7}, one has
	\begin{equation*}
		\limsup_{n\to +\infty} \int_{\mathcal{Z}_{f}(\Omega_{n})} \varphi(t) H_{1}(\mathrm{d}t) \leq \liminf_{g \to f} \int_{\mathcal{Z}_{g}(\Omega)}\varphi(t) H_{1}(\mathrm{d}t).
	\end{equation*}

	The left hand part converges, in virtue of the dominated convergence theorem, to
	\begin{equation*}
		\int_{\mathcal{Z}_{f}(\Omega)}\varphi(t) H_{1}(\mathrm{d}t).
	\end{equation*}
	Hence,
	\begin{equation*}
		\int_{\mathcal{Z}_{f}(\Omega)}\varphi(t) H_{1}(\mathrm{d}t)\leq \liminf_{g \to f} \int_{\mathcal{Z}_{g}(\Omega)}\varphi(t) H_{1}(\mathrm{d}t)
	\end{equation*}
	Note that this inequality is always true and is independent of condition \eqref{equation 6.6.1}. The latter condition will be needed for the converse inequality.\medbreak

	This time, we let $\Omega^{n}$ be a sequence of non-increasing convex open sets such that $\overline{\Omega^{n+1}} \subset \Omega^{n}$ and
	\begin{equation*}
		\bigcap_{n \geq 1} \Omega^{n}=\overline{\Omega}.
	\end{equation*}
	Following the preceding reasoning, one has
	\begin{align*}
		\limsup_{g \to f}\int_{\mathcal{Z}_{g}(\Omega)} \varphi(t)H_{1}(\mathrm{d}t) &\leq \liminf_{n\to +\infty}\int_{\mathcal{Z}(\Omega^{n})} \varphi(t)H_{1}(\mathrm{d}t)\\
		&=\int_{\mathcal{Z}(\overline{\Omega})} \varphi(t)H_{1}(\mathrm{d}t)\\
		\text{(condition \ref{equation 6.6.1}) }&=\int_{\mathcal{Z}(\Omega)} \varphi(t)H_{1}(\mathrm{d}t).
	\end{align*}

	The two preceding inequality immediately implies that
	\begin{equation*}
		\lim_{g \to f}\Gamma(g) = \Gamma(f)
	\end{equation*}
	which finishes the proof.
\end{proof}

The latter part of this section is thus devoted to the proof of Proposition \ref{proposition 6.7}. The ideas involved into proving this result are summarised hereafter.\medbreak

First, we establish for $g \in \mathcal{C}^{1}(\overline{\Omega})$ sufficiently close to $f$ a one-to-one local correspondence between the  connected components of $f^{-1}(0)$ and the ones of $g^{-1}(0)$. What is meant with local correspondence is the existence of a covering of $f^{-1}(0)$  by open sets $V_{i}$ such that $\mathcal{Z}_{g}(V_{i})$ consists in exactly one connected component $\gamma_{g}$. If we reformulate the latter statement, this means that the set $V_{i}$ is crossed by exactly one nodal line of $g$. The rest of the proof consists into estimating the difference between $\gamma_{g}$ and $\gamma_{f}$ and showing that the latter is negligible as $g$ tends to $f$. To translate this local result into a global one, a partition of unity argument is later invoked.

\subsubsection*{Technical lemmas}

The first step of the proof consists of proving this local correspondence. We will need the following lemma, which establishes that the zeros of $f$ and $g$ cannot be too far away as long as $g$ is sufficiently close to $f$. \medbreak

First, we define some notations. As $f$ is regular, there exists an open set $W \subset \Omega$ containing $\mathcal{Z}(\overline{C})$ such that
\begin{equation*}
	\inf_{t \in W} |\nabla f(t)| \geq m
\end{equation*}
with $m>0$.\medbreak

On the other hand, the uniform continuity on $\overline{U}$ ensures that
\begin{equation*}
	|\nabla f(t)-\nabla f(s)| \leq \frac{m}{4}, \quad |s-t| \leq \delta.
\end{equation*}
for some $\delta>0$.

\begin{lemma}\label{lemma 6.8}
	Let $\varepsilon < \delta m/8$ and $g \in \mathcal{C}^{1}(\overline{U})$ such that $||f-g||_{1} \leq \varepsilon$.	If $f(t)=0$ and $|\partial_{1}f(t)| \geq m/2$, then there exists $|\eta | \leq 8\varepsilon/m$ such that
	\begin{equation*}
		g(t_{1}+\eta, t_{2}) = 0.
	\end{equation*}
\end{lemma}

In particular, if $f(t)=0$, either $|\partial_{1}f(t)|$ or $|\partial_{2}f(t)|$ is no less than $m/2$ and the preceding lemma can be applied.

\begin{corollary}\label{corollary 6.9}
	There is a neighbourhood $V_{f}$ of $f \in \mathcal{C}^{1}(\overline{U})$  such that any $t \in \mathcal{Z}_{f}(\overline{C})$ can be associated with some $s \in \mathcal{Z}_{g}(D)$ such that $|s-t|\leq 8\varepsilon/m$.

\end{corollary}

\begin{proof}
	Without loss of generality, assume that $\partial_{1}f(t)>0$ and consider the function
	\begin{equation*}
		F(\omega):=f(t_{1}+\omega,t_{2}).
	\end{equation*}
	whose derivatives is given by
	\begin{equation*}
		F'(\omega)=\partial_{1}f(t_{1}+\omega,t_{2})
	\end{equation*}
	For $|\omega| \leq \delta$, the triangle inequality ensures that
	\begin{align*}
		F'(\omega) \geq  \frac{m}{4},
	\end{align*}
	so that $f(t_{1}+8\varepsilon/m,t_{2})$ and $f(t_{1}-8\varepsilon/m,t_{2})$ have opposite sign and their respective absolute value are at least $2\varepsilon$. Thus, $g(t_{1}+8\varepsilon/m,t_{2})$ and $g(t_{1}-8\varepsilon/m,t_{2})$ have opposite sign and the conclusion follows from the Intermediate Value Theorem.
\end{proof}

\begin{lemma}\label{lemma 6.10}
	There is a finite open cover $V_{i}$ of $\mathcal{Z}_{f}(\overline{C})$ and a neighbourhood $V_{f}$ of $f$ in $\mathcal{C}^{1}(\overline{U})$  such that for any $g \in V_{f}$, the following assertions are satisfied
	\begin{enumerate}
		\item $\mathcal{Z}_{g}(V_{i})$ has exactly one connected component.
		\item
		\begin{equation*}
			\mathcal{Z}_{g}(C) \subset \bigcup_{i} V_{i} \subset D.
		\end{equation*}
	\end{enumerate}
	In particular,
	\begin{equation*}
		\mathcal{Z}_{g}(C) = \bigcup_{i} \mathcal{Z}_{g}(V_{i}) \subset \mathcal{Z}_{g}(D).
	\end{equation*}
\end{lemma}

\begin{proof}
	Let $d_{g}$ be the minimal distance between two different connected components of $\mathcal{Z}_{g}(\overline{D})$. If there is no or only one connected component, $d_{g}$ is arbitrarily set to $1$.\medbreak

	As $\overline{D}$ is compact, two different connected different components of $f$ are always at a non zero distance and this bound holds uniformly in the sense there is a neighbourhood $V_{f}$ of $f$ such that
	\begin{equation*}
		d:=\inf_{g \in V_{f}} d_{g}>0.
	\end{equation*}

	The proof of this result is omitted here but it requires the use of a quantitative version of the implicit theorem. Such statement yields an explicit neighbourhood, whose size depends only the first derivatives of the function, on which the conclusion of the implicit theorem holds. In particular, in the latter neighbourhood, the zero set of the related function is composed of exactly one connected component. Applying this result to a function $g$ in a small neighbourhood of $f$, one sees that the size of the neighbourhood can be chosen so that it only depends on $f$. In particular, this prevents $\mathcal{Z}_{g}(\overline{D})$ from having two close connected components.\medbreak

	To construct the sets $V_{i}$, we choose a finite cover of $\mathcal{Z}_{f}(\overline{C})$ with open balls whose diameter is less than $d$ and such that $\mathcal{Z}_{f}(V_{i})$ is non empty. By definition, $\mathcal{Z}_{g}(V_{i})$ has at most one connected component for $g \in V_{f}$. On the other hand, Corollary \ref{corollary 6.9} ensures there is at least one. This finishes the proof.
\end{proof}

In virtue of Corollary \ref{lemma 6.10}, the unique connected component of $\mathcal{Z}_{g}(V_{i})$ can be parametrised by $(\gamma_{g,i}(t),t)$ or  $(t,\gamma_{g,i}(t))$ depending on whether $|\partial_{1}f|\geq m/2$ or $|\partial_{2}f|\geq m/2$ on $V_{i}$. Note that it is always possible to choose $V_{i}$ so that at least one of these two assumptions hold. Without loss of generality, we will only consider parametrisation of the second form.

\begin{lemma}\label{lemma 6.11}
	As $g \to f$, $\gamma_{i,g}$ converges to $\gamma_{i,f}$ for the topology of uniform convergence of the first derivatives on compact sets.
\end{lemma}

\begin{proof}
	The uniform convergence of $\gamma_{i,g}$ to $\gamma_{i,f}$ is a straightforward consequence of Corollary \ref{corollary 6.9} which bounds the distance between the cancellation points of $f$ and $g$.

	As for the convergence of derivatives, it follows from the formula
	\begin{equation*}
		\gamma_{i,g}'(t)=-\frac{\partial_{1}g(t,\gamma_{i,g}(t))}{\partial_{2}g(t,\gamma_{i,g}(t))} \to -\frac{\partial_{1}f(t,\gamma_{i,f}(t))}{\partial_{2}f(t,\gamma_{i,f}(t))} = \gamma_{i,f}'(t)
	\end{equation*}
\end{proof}

With Lemma \ref{lemma 6.11} proved, all the necessary tools for the proof of Proposition \ref{proposition 6.7} has been introduced.

\begin{proof}[Proof of Proposition \ref{proposition 6.7}]
	Consider a partition of unity associated with the open sets $V_{i}$, that is a a family of continuous functions $0 \leq \phi_{i} \leq 1$ with compact support included in $V_{i}$ satisfying
	\begin{equation*}
		\sum_{i=1}^{n} \phi_{i}(t) \leq 1, \quad t \in D,
	\end{equation*}
	with equality when $t \in C$.\medbreak

	Let $f_{n}$ be a sequence a sequence of functions converging to $f$ in $\mathcal{C}^{1}(\overline{U})$.
	The positivity of $\varphi$ implies that
	\begin{align*}
		\int_{\mathcal{Z}_{f_{n}}(C)} \varphi(t) H_{1}(dt)&= \sum_{i} \int_{\mathcal{Z}_{f_{n}}(V_{i}\cap C)} \phi_{i}(t)\varphi(t) H_{1}(\mathrm{d}t)\\
		&\leq \sum_{i}\int_{\mathcal{Z}_{f_{n}}(V_{i})}\phi_{i}(t)\varphi(t) H_{1}(\mathrm{d}t)\\
		&=\sum_{i} \int \phi(t, \gamma_{i,f_{n}}(t)) \varphi(t, \gamma_{i,f_{n}}(t))\sqrt{1+\gamma_{i,f_{n}}'(t,\gamma_{i,f_{n}}(t))^{2}} \mathrm{d}t.
	\end{align*}
	Lemma \ref{lemma 6.11} implies that the latter converges to
	\begin{align*}
		&\int \phi(t, \gamma_{i,f}(t)) \varphi(t, \gamma_{i,f}(t))\sqrt{1+\gamma_{i,f}'(t,\gamma_{i,f}(t))^{2}} \mathrm{d}t\\
		=& \sum_{i}\int_{\mathcal{Z}_{f}(V_{i})}\phi_{i}(t)\varphi(t) H_{1}(\mathrm{d}t)\\
		\leq& \sum_{i}\int_{\mathcal{Z}_{f}(D)}\phi_{i}(t)\varphi(t) H_{1}(\mathrm{d}t)\\
		\leq& \int_{\mathcal{Z}_{f}(D)}\varphi(t) H_{1}(\mathrm{d}t)
	\end{align*}

	Hence,
	\begin{equation*}
		\limsup_{g \to f} \int_{\mathcal{Z}_{g}(C)} \varphi(t) H_{1}(\mathrm{d}t) \leq \int_{\mathcal{Z}_{f}(D)} \varphi(t) H_{1}(\mathrm{d}t),
	\end{equation*}
	which ultimately concludes the proof of Proposition \ref{proposition 6.7}.
\end{proof}

\bibliographystyle{alpha}
\bibliography{Bibliography.bib}

\end{document}